\documentclass[12pt]{amsart}
\usepackage{latexsym,enumerate}
\usepackage{amssymb, xcolor,mathrsfs}
\usepackage{amsmath,amsthm,amsfonts,amssymb,latexsym, mathabx}
\usepackage[pagebackref]{hyperref}

\newtheorem{theorem}{Theorem}[section]
\newtheorem{lemma}[theorem]{Lemma}
\newtheorem{proposition}[theorem]{Proposition}
\newtheorem{corollary}[theorem]{Corollary}

\newtheorem{consequence}[theorem]{Consequence}
\newtheorem{conjecture}[theorem]{Conjecture}

\newtheorem{theorema}{Theorem}

\newtheorem{conj}[theorema]{Conjecture}

\theoremstyle{definition}
\newtheorem{definition}[theorem]{Definition}

\newtheorem{remark}[theorem]{Remark}

\newtheorem*{acknowledgement}{Acknowledgement}

\numberwithin{equation}{section}

\newcommand{\GL}{{\mathrm {GL}}}

\newcommand{\SL}{{\mathrm {SL}}}
\newcommand{\PSL}{{\mathrm {PSL}}}

\newcommand{\GU}{{\mathrm {GU}}}

\newcommand{\PSU}{{\mathrm {PSU}}}

\newcommand{\SO}{{\mathrm {SO}}}

\newcommand{\Sp}{{\mathrm {Sp}}}
\newcommand{\PSp}{{\mathrm {PSp}}}

\newcommand{\Aut}{{\mathrm {Aut}}}

\newcommand{\Irr}{{\mathrm {Irr}}}

\newcommand{\ord}{{\mathrm {ord}}}

\newcommand{\Ind}{{\mathrm {Ind}}}

\newcommand{\Syl}{{\mathrm {Syl}}}

\newcommand{\lcm}{{\mathrm {lcm}}}
\newcommand{\Gal}{{\rm Gal}}

\newcommand{\Res}{{\mathrm {Res}}}

\newcommand{\CC}{{\mathbb C}}

\newcommand{\QQ}{{\mathbb Q}}
\newcommand{\ZZ}{{\mathbb Z}}

\newcommand{\FF}{{\mathbb F}}

\newcommand{\wt}{\widetilde}

\newcommand{\height}{\mathbf{ht}}
\newcommand{\lev}{\mathrm{\mathbf{lev}}}
\newcommand{\Bl}{\mathrm{\mathrm{Bl}}}

\newcommand{\bC}{{\mathbf{C}}}
\newcommand{\bG}{{\mathbf{G}}}

\newcommand{\bO}{{\mathbf{O}}}
\newcommand{\bN}{{\mathbf{N}}}
\newcommand{\bZ}{{\mathbf{Z}}}
\newcommand{\Al}{\textup{\textsf{A}}}

\newcommand{\tw}[1]{{}^#1}

\def\nor{\trianglelefteq\,}
\def\irr#1{{\rm Irr}(#1)}

\def\zent#1{{\bf Z}(#1)}

\def\cent#1#2{{\bf C}_{#1}(#2)}
\def\ker#1{{\rm Ker}(#1)}

\newcommand{\type}{\operatorname}

\def\hc#1#2{{\operatorname{R}}_{#1}^{#2}}

\usepackage{hyperref}

\begin{document}

\title[The $p$-rationality of $p$-height-zero characters]
{The $p$-rationality of $p$-height-zero characters}

\author[Nguyen N. Hung]{Nguyen N. Hung}
\address[Nguyen N. Hung]{Department of Mathematics, The University of Akron, Akron,
OH 44325, USA} \email{hungnguyen@uakron.edu}

\author[A. A. Schaeffer Fry]{A. A. Schaeffer Fry}
\address[A. A. Schaeffer Fry]{Dept. Mathematics, University of Denver, Denver, CO 80210, USA}
\email{mandi.schaefferfry@du.edu}


\thanks{The first author gratefully acknowledges the support of an AMS-Simons Research Enhancement Grant and UA Faculty Research Grant FRG 1747. The second author gratefully acknowledges support from the National Science Foundation, Award No. DMS-2100912,  and her former institution, Metropolitan State University of Denver, which held the award and allowed her to serve as PI}


\begin{abstract} We propose and present evidence for a conjectural global-local phenomenon concerning the
$p$-rationality of $p$-height-zero characters. Specifically, if
$\chi$ is a height-zero character of a finite group $G$ and $D$ is a
defect group of the $p$-block of $G$ containing $\chi$, then the
$p$-rationality of $\chi$ can be captured inside the normalizer
$\bN_G(D)$.
\end{abstract}

\maketitle

\tableofcontents


\section{Introduction}
In recent years, there has been a growing interest in the study of rationality properties of characters of finite groups and their relationship to global-local properties, which relate the character theory of a finite group to $p$-local subgroups for a prime $p$.  In this paper, we are concerned with the
$p$-rationality of ($p$-)height-zero characters of finite groups.

Let $B$ be a $p$-block of a finite group $G$ and let $\Irr(B)$ denote the set of ordinary irreducible characters of $B$. The
\emph{($p$-)height} of a character $\chi\in\Irr(B)$ is given by \[{\height(\chi):=\nu(\chi(1))-\min_{\psi\in\Irr(B)}\{\nu(\psi(1))\}},\] where
$\nu:=\nu_p$ is the usual $p$-adic valuation function. We say that $\chi$ is
\emph{height-zero} if
$\height(\chi)=0$. In other words, the height-zero characters of $B$
are those characters in $B$ whose degrees have the minimal possible $p$-part.

To measure how $p$-rational (or $p$-irrational) a character $\chi$
is, one considers the $p$-part of the conductor of its values
$\{\chi(g): g\in G\}$. Recall that every character value is a
certain sum of roots of unity. Such a sum is called a \emph{cyclotomic
integer}. The conductor $c(\mathcal{S})$ of a collection
$\mathcal{S}$ of cyclotomic integers is the smallest positive
integer $n$ such that $\mathcal{S}\subseteq \QQ(\exp(2\pi i/n))$. For $\chi\in\Irr(G)$, we write $c(\chi):= c(\{\chi(g):g\in G\})$ and call this the \emph{conductor of $\chi$}. The
so-called \emph{$p$-rationality level} of $\chi$ is defined as
\[
\lev(\chi):=\nu(c(\chi)).
\]

We put forward the following, which proposes that the
$p$-rationality level of a height-zero character can be captured
inside a local subgroup, namely the defect normalizer.

\begin{conj}\label{conj:main}
Let $p$ be a prime, $G$ a finite group, and $\chi$ be a height-zero
character in a block $B$ of $G$ with $\lev(\chi)\geq 2$. Suppose
that $D$ is a defect group of $B$. Then
\[\lev(\chi)=\lev(\chi_{\bN_G(D)}).\]
\end{conj}

\begin{remark}\label{remark-1}
For $g\in G$, let $\lev(\chi(g)):=\nu_p(c(\chi(g)))$ -- the $p$-rationality level of $\chi(g)$. It is easy to see that
\[\lev(\chi)=\max_{g\in G}\{\lev(\chi(g))\},\]
 so there exists $g\in G$ such that $\lev(\chi(g))=\lev(\chi)$.
That is, there exists an element in the group that captures the
$p$-rationality of $\chi$. Conjecture~\ref{conj:main} simply
claims that such an element can be found in $\bN_G(D)$.
\end{remark}

Height-zero characters are well known for their nice behavior with
respect to the global-local principal. Among the first observations of this was in Brauer's height zero
conjecture, recently resolved in \cite{Malle-et-al24}, which states
that all irreducible characters in a block $B$ have height zero if
and only if the defect groups of $B$ are abelian. Another example is the celebrated
Alperin--McKay conjecture. (See, e.g. \cite[Conjecture 9.5]{Navarro18}. See also \cite{Ruhstorfer}, where the conjecture was recently proven for $p=2$.) The Alperin--McKay conjecture asserts that if $b$ is the block of $\bN_G(D)$
corresponding to $B$ in Brauer's first main correspondence, then
there exists a bijection between the height-zero characters in
$B$ and those in $b$. Conjecture~\ref{conj:main} offers another global-local phenomenon
for height-zero characters. In fact, we observe in Section~\ref{sec:AMN-conjecture} a relationship between Conjecture~\ref{conj:main} and the well-known Alperin--McKay--Navarro conjecture, which refines the Alperin--McKay conjecture to further include the action of Galois automorphisms.

Conjecture~\ref{conj:main} is inspired by Navarro-Tiep's conjecture \cite[Conjecture~C]{Navarro-Tiep21}.
In what follows, $\Irr(G)$ denotes the set of irreducible characters
of $G$ and $\Irr_{p'}(G)$ the subset of $\Irr(G)$ consisting of
characters of degree not divisible by $p$. Furthermore, for any positive integer $n$, we use $\QQ_n$ to denote the $n$-th cyclotomic field $\QQ_n:=\QQ(\exp(2\pi i/n))$.

\begin{conj}[\cite{Navarro-Tiep21}, Conjecture~C]\label{conj:Navarro-Tiep1}
Let $p$ be a prime, $G$ a finite group, $P\in\Syl_p(G)$, and
$\chi\in \Irr_{p'}(G)$ with $\lev(\chi)\geq1$. Then
\[\QQ_{p^{\lev(\chi)}}=\QQ_p(\chi_P).\]
\end{conj}

Conjecture~\ref{conj:Navarro-Tiep1} suggests that if a $p'$-degree character $\chi$
has $p$-rationality level at least 2, then its level remains
unchanged when restricted to a Sylow $p$-subgroup:
$\lev(\chi)=\lev(\chi_{P})$. Recall that a $p'$-degree character is a height-zero character lying in a block of maximal defect, which means that the defect groups
are the Sylow $p$-subgroups of $G$. In such a case,
Conjecture~\ref{conj:main} only asserts that
$\lev(\chi)=\lev(\chi_{\bN_G(P)})$. However, it is important to note
that for height-zero characters in general, $\lev(\chi)$ does not
always equal $\lev(\chi_{D})$, see the examples in Subsection~\ref{examples}. We refer the reader to \cite{Navarro-Tiep21,Isaacs-Navarro22} for
further discussion on
Conjecture~\ref{conj:Navarro-Tiep1}.

What evidence do we have for Conjecture~\ref{conj:main}? Our first
main result confirms the cyclic-defect case.

\begin{theorema}\label{mainthm:cyclic-defect-case}
Let $p$ be a prime and $G$ a finite group. Let $B\in \Bl(G)$ be a
$p$-block of $G$ with cyclic defect group $D$ and $\chi\in\Irr(B)$.
Then $\lev(\chi)=\lev(\chi_{\bN_G(D)})$.
\end{theorema}

\begin{remark}
The assumption on $\lev(\chi)$ in Conjecture~\ref{conj:main} is
essential. There are many examples with $\lev(\chi)=1$ but
$\lev(\chi_{\bN_G(D)})=0$, see again Subsection~\ref{examples}.
Theorem~\ref{mainthm:cyclic-defect-case}, however, shows that this
cannot occur when a defect group $D$ is cyclic.
\end{remark}

Our next result solves Conjecture~\ref{conj:Navarro-Tiep1} for
prime-degree characters, and therefore confirms
Conjecture~\ref{conj:main} for  characters whose degree is a prime different
from $p$.

\begin{theorema}\label{mainthm:prime-degree-case}
Let $p$ be a prime and $G$ a finite group. Let $\chi\in \Irr(G)$ be
of prime degree not equal to $p$ with $\lev(\chi)\geq 2$. Let
$P\in\Syl_p(G)$. Then $\QQ_{p^{\lev(\chi)}}=\QQ_p(\chi_P)$. In
particular, $\lev(\chi)=\lev(\chi_{\bN_G(P)})=\lev(\chi_{P})$.
\end{theorema}

The proof of Theorem \ref{mainthm:cyclic-defect-case} is based on Dade's cyclic-defect theory \cite{Dade66,Dade96}. When the defect groups of $B$ are cyclic, the set $\Irr(B)$ is naturally partitioned into two types: \emph{exceptional} characters and \emph{non-exceptional} characters, see Subsection~\ref{subsec:generalities}. While the characters of the latter type are always $p$-rational, we show that the $p$-rationality level of an  exceptional character aligns with that of its associated (linear) character of the defect group $D$. Another key step is to show that if an element $g\in G$ captures the $p$-rationality of a character $\chi$ in a block with cyclic defect groups, then its $p$-part $g_p$ generates a defect group of the block. This is done in Section~\ref{sec:cyclic-defect}.

The proof of Theorem~\ref{mainthm:prime-degree-case} is divided into two fundamentally different cases, depending on whether the character in question is primitive or imprimitive. The imprimitive case builds on ideas from the recent solution of Conjecture~\ref{conj:Navarro-Tiep1} for
$p$-solvable groups \cite{Isaacs-Navarro22}, as detailed in Section~\ref{sec:imprimitive}. In contrast, the primitive case is reduced to analyzing the values of prime-degree characters of quasisimple groups, which is addressed in Sections~\ref{sec:quasisimple} and \ref{sec:primitive}.

Subsection~\ref{sec:further-examples} provides additional evidence supporting Conjecture~\ref{conj:main} for certain almost quasisimple groups. Finally, Section~\ref{sec:AMN-conjecture} discusses some consequences and examples related to Conjecture~\ref{conj:main}, along with its connection to the well-known Alperin-McKay-Navarro conjecture.

\begin{acknowledgement}
The authors are grateful to Thomas Breuer, Gabriel Navarro, and Benjamin Sambale for several helpful conversations on topics related to this work. Part of this work was
completed while the first author was visiting the Department of
Mathematics at the University of Denver. He thanks the department
and faculty members Alvaro Arias, Mandi Schaeffer Fry, Michael
Kinyon, and Petr Vojtechovsky for their hospitality.
\end{acknowledgement}



\section{Galos automorphisms and $p$-rationality level}\label{sec:Galauts}

Here we discuss briefly the relationship between the $p$-rationality level of a character and the action of Galois automorphisms. The notation here will often be referred to throughout.

Let $G$ be a finite group and suppose that $|G|=n=p^bm$ with $(p,m)=1$. Let
$\mathcal{G}:=\mathrm{Gal}(\QQ_n/\QQ)$. Then
\[\mathcal{G}\cong\mathcal{I}\times \mathcal{K},\] where
\[\mathcal{I}=\mathrm{Gal}(\QQ_n/\QQ_m) \text{ and }
\mathcal{K}=\mathrm{Gal}(\QQ_n/\QQ_{p^b})\] are the subgroups of
$\mathcal{G}$ of those automorphisms fixing $p'$-roots and $p$-power
roots, respectively, of unity. Let \[\mathcal{H}:=\mathcal{I}\times
\langle \sigma\rangle,\] where $\sigma\in \mathcal{K}$ is such that its
restriction to $\QQ_m$ is the Frobenius automorphism $\zeta \mapsto
\zeta^p$. The group $\mathcal{H}$ is an important ingredient in the McKay--Navarro and Alperin--McKay--Navarro Conjectures \cite{Navarro04}, which we will discuss further in Section \ref{sec:AMN-conjecture}.

It is well known that the Galois group $\mathcal{G}$
permutes the $p$-blocks of $G$. Let $B$ be a $p$-block of $G$ and
$\mathcal{H}_B$ be the subgroup of $\mathcal{H}$ fixing $B$. Since
$\mathcal{I}$ point-wisely fixes $\QQ_m$, it fixes every Brauer
character and thus every block of $G$. In particular,
\[\mathcal{I}\leq \mathcal{H}_B.\]

We define $\mathcal{I}':=\mathcal{I}$ if $p=2$ and
$\mathcal{I}':=\mathrm{Gal}(\QQ_{n}/\QQ_{pm})$ if $p>2$. Note that
$\mathcal{I}'$ is the Sylow $p$-subgroup of $\mathcal{I}$. Note also that if $H_1$ and $H_2$ are groups with orders dividing $n$, then characters $\chi_1$ of $H_1$ and $\chi_2$ of $H_2$ have the same $p$-rationality level whenever they have the same stabilizer under $\mathcal{I}$. Further, the same can be said, replacing $\mathcal{I}$ with $\mathcal{I}'$, with the added assumption that  $\lev(\chi_i)\geq 1$ for $i=1,2$ if $p>2$.

In fact, as pointed out in \cite[Section~4]{Navarro-Tiep21}, there is one particular element of $\mathcal{I}'$ that captures this behavior. 
Namely, for $e\in\ZZ_{\geq 1}$, let $\sigma_e$ be the element of
 $\mathcal{I}'$ mapping any $p$-power root of unity
$\omega$ to $\omega^{1+p^e}$. The Galois automorphism $\sigma_e$ has been seen to play a pivotal role in consequences of the McKay--Navarro conjecture predicting global-local properties of finite groups, and it turns out that the stability of a character under $\sigma_e$ is closely
tied to its $p$-rationality level (see \cite[Lemma~4.1]{Navarro-Tiep21}).


\section{Blocks of cyclic defect}\label{sec:cyclic-defect}

The goal of this section is to prove Theorem~\ref{mainthm:cyclic-defect-case}.

\subsection{Small-defect blocks} We begin with an elementary upper bound for the $p$-rationality
level in terms of a defect group, which allows us to easily control
the level of characters in blocks of small defect.

Recall that if $B$ is a $p$-block of a finite group $G$ then its \emph{defect}
$d(B)$ is the nonnegative integer
\[{d(B)}:=\nu(|G|)-\min_{\psi\in\Irr(B)}\{\nu(\psi(1))\}.\] Moreover, the order of any defect group $D$ of $B$ is $|D|=p^{d(B)}$. We will also denote by $\mathrm{exp}(D)$ the exponent of $D$.

Throughout, for an integer $n$, we will write $n_p$ and $n_{p'}$ for its $p$- and $p'$-parts, respectively, so that $n=n_pn_{p'}$, $n_p$ is a power of $p$, and $(p, n_{p'})=1$. Similarly, for an element $g$ of a finite group $G$, we will write $g_p$ and $g_{p'}$ for the (unique) elements such that $g=g_pg_{p'}$ with $|g_p|=|g|_p$ and $|g_{p'}|=|g|_{p'}$.  

\begin{lemma}\label{lem:1}
Let $\chi\in\Irr(G)$ and $B$ the $p$-block of $G$ containing $\chi$.
Let $D$ be a defect group of $B$. Then $\lev(\chi)\leq
\nu(\exp(D))$. In particular, $\lev(\chi)\leq d(B)$.
\end{lemma}

\begin{proof}
Let $g\in G$ with $\chi(g)\neq 0$. Then $g_p$ 
belongs to a conjugate of $D$ by \cite[Corollary~5.9]{nbook}.
Therefore $|g|_p\leq \exp(D)$, and we have
\[\chi(g)\in\QQ_{|g|}\subseteq \QQ_{|g|_p|g|_{p'}}\subseteq
\QQ_{\exp(D)|G|_{p'}}.\] We now have $\QQ(\chi)\subseteq
\QQ_{\exp(D)|G|_{p'}}$, which implies that $c(\chi)$ divides
$\exp(D)|G|_{p'}$, and the lemma follows.
\end{proof}

\begin{corollary}\label{thm:defect-zero-char}
Let $\chi\in\Irr(G)$ belong to a $p$-block of defect 0, or defect 1
if $p=2$. Then $\lev(\chi)=0$.
\end{corollary}

\begin{proof}
This follows from Lemma~\ref{lem:1}. Note that $\QQ_n=\QQ_{2n}$ for
odd $n$, so there are no characters of $2$-rationality level $1$.
\end{proof}

\begin{lemma}\label{thm:block-defect-one-odd-p}
Let $p$ be an odd prime. Suppose $\chi\in\Irr(G)$ belongs to a
$p$-block of defect one with a defect group $D$. Then
$\lev(\chi)\in\{0,1\}$ and $\lev(\chi)=\lev(\chi_{\bN_G(D)})$.
\end{lemma}

\begin{proof}
The first conclusion again follows from Lemma~\ref{lem:1},  so it
remains to prove the second part. We have $\chi(1)_p=|G|_p/p$ and
$\height(\chi)=0$. $p$-Blocks of defect one have been fully
described in early work of Brauer \cite{Brauer42a,Brauer42b} (see
also \cite[Chapter 11]{nbook}). Let $K:=\bO_{p'}(\bN_G(D))$. Let
$b\in \Bl(\bC_G(D))$ be a root of $B$ and $\xi\in \Irr(K)$ be the
restriction (to $K$) of the canonical character in $\Irr(b)$ of $B$.
In fact, $\xi$ is the unique Brauer character in $b$. Let $E$ be the
subgroup of $\bN_G(D)$ fixing $b$ and $\overline{E}:=E/\bC_G(D)$ the
inertial quotient of $B$ and $e:=|\overline{E}|$. The block $B$ then
contains precisely $e+(p-1)/e$ ordinary irreducible characters, where $e$
of these are $p$-rational
and therefore trivially satisfy the stated equality.

Suppose that $\chi$ is one of the remaining $(p-1)/e$ other characters, a
so-called exceptional character. In this case, there exists
$\lambda\in \Irr(D)-\{1_D\}$ and $\epsilon\in\{\pm 1\}$ such that
\[
\chi(hk)=\epsilon(\lambda\times \xi)^{\bN_G(D)}(hk)
\]
for every $h\in D-\{1\}$ and $k\in K$, by \cite[Chapter 11]{nbook}. Note that $\chi(g)=0$
whenever $g_p$ is not conjugate to an element in $D$. Also, for each
$h\in D-\{1\}$, a $p'$-element in $\bC_G(h)$ must be inside $K$.
Therefore these elements $hk$ capture all the non-zero values of
$\chi$. It follows that $\QQ(\chi)=\QQ(\chi_{D\times K})$, and thus
$\QQ(\chi)=\QQ(\chi_{\bN_G(D)})$, as desired.
\end{proof}


\subsection{Generalities on blocks with cyclic defect groups}\label{subsec:generalities} To prove Theorem~\ref{mainthm:cyclic-defect-case} for blocks of
larger cyclic defect, we need to recall some basics on
cyclic-defect theory, and refer the reader to \cite{Dade66}
and \cite[Chapter VII]{Feit82} for more details. The theory,
developed by E.~Dade in the sixties, generalizes Brauer's work on
defect-one blocks mentioned above.

Let $B$ be a block of a finite group $G$ with cyclic defect group
$D$ of order $p^a$ ($a\geq 1$). Let $B_0\in\Bl(\bN_G(D))$ be the
Brauer correspondent of $B$ and $b_0\in\Bl(\bC_G(D))$ be a root of
$B$; i.e. $b_0^{\bN_G(D)}=B_0$. Let $C:=\bC_G(D)$ and let $E$ be the subgroup of
$\bN_G(D)$ fixing $b_0$. The inertial quotient $E/C$ is then a
cyclic group of $p'$-order acting Frobeniusly on $D$ (as well as
$\Irr(D)$). Let $\Lambda$ be a complete set of representatives of
the action of $E$ on $\Irr(D)-\{1_D\}$. Then $|\Lambda|=(p^a-1)/e$,
where $e:=|E/C|$.

If $|\Lambda|=1$, then $D$ must have order $p$ and $B$ has
precisely $e+1$  irreducible  ordinary characters. Suppose that
$|\Lambda|>1$. Then $\Irr(B)$ is partitioned into two naturally
defined subsets $\Irr_{nex}(B)$ and $\Irr_{ex}(B)$. The former
consists of precisely $e$ \emph{non-exceptional characters}
$\{X_1,...,X_e\}$. The latter consists of precisely $|\Lambda|$
\emph{exceptional characters}, which are naturally labeled by the
members of $\Lambda$: \[\Irr_{ex}(B)=\{X_\lambda\mid
\lambda\in\Lambda\}.\]

As noted in \cite[p.~1135]{Navarro04}, the group $\mathcal{H}_B$
permutes the exceptional/non-exceptional characters among themselves.

For $0\leq i\leq a$, let $D_i$ be the (unique) subgroup of $D$ containing the elements of order at most $p^{a-i}$; that is, 
\[D_i \text{ is the subgroup of } D \text{ with }
|D:D_i|=p^i.\] Let \[C_i:=\bC_G(D_i) \text{ and } N_i:=\bN_G(D_i).\] Assume now that
$0\leq i\leq a-1$. By \cite[Proposition 1.6]{Dade66},  the block
$b_i:=(b_0)^{C_i}\in\Bl(C_i)$ contains a unique Brauer character,
say $\varphi_i$. By \cite[Corollary 1.9]{Dade66}, for $x\in
D_i-D_{i+1}$, and $y$ a $p$-regular element of $C_i$, we have  for each $\lambda\in\Lambda$,
\begin{equation}\label{eq:2}
X_\lambda(xy)=\frac{\delta}{|C_i|} \sum_{h\in N_i}
\lambda^h(x)(\varphi_i)^h(y),
\end{equation}
for some $\delta\in\{\pm 1\}$ depending only on $i$; and for each $1\leq j\leq e$,
\[
X_j(xy)=\frac{\pm1}{e|C_i|} \sum_{h\in N_i} (\varphi_i)^h(y).
\]

We remark that, if $\chi\in\Irr(G)$ and $g\in G$, then $\chi(g)$ is a
sum of $|g|$-th roots of unity. The nonexceptional characters
$X_j$ are therefore always $p$-rational.

\subsection{Proof of Theorem~\ref{mainthm:cyclic-defect-case}}
We are ready to prove Theorem~\ref{mainthm:cyclic-defect-case}, which we now restate.

\begin{theorem}\label{thm:cyclic-defect-case}
Let $p$ be a prime and $G$ a finite group. Let $B\in \Bl(G)$ be a
$p$-block of $G$ with cyclic defect group $D$. Then
$\lev(\chi)=\lev(\chi_{\bN_G(D)})$ for every $\chi\in\Irr(B)$.
\end{theorem}

\begin{proof} The case of defect zero follows from Corollary~\ref{thm:defect-zero-char}. We may assume that $|D|>1$.

We shall follow the notation described above. If $|\Lambda|=1$ then,
as mentioned already, the block $B$ must have defect one, and we are
done by Corollary~\ref{thm:defect-zero-char} and
Lemma~\ref{thm:block-defect-one-odd-p}. So we assume from now on
that $|\Lambda|>1$. If $\chi\in\Irr_{nex}(B)$ is a  non-exceptional
character of $B$, then $\chi$ is $p$-rational, and thus there is
nothing to prove. We therefore assume furthermore that $\chi$ is one
of the exceptional characters $X_\lambda$ for some $ \lambda\in\Lambda$.

We claim that \[\text{if } \lev(X_\lambda)\geq 1, \text{ then }
\lev(X_\lambda)=\lev(\lambda).
\]

Let $\widetilde{D}$ be the unique subgroup of $D$ of order $p$ and
set $\widetilde{N}:=\bN_G(\widetilde{D})$. Let
$\widetilde{B}=(b_0)^{\widetilde{N}}=(B_0)^{\widetilde{N}}$, which
is a block of $\widetilde{N}$ that has the same defect group $D$ and
Brauer correspondent $B_0$ as $B$. Dade proved in
\cite[Lemma~4.9]{Dade96} that the exceptional characters of
$\widetilde{B}$ can be labeled by
$\Irr_{ex}(\widetilde{B})=\{\widetilde{X}_\lambda\mid
\lambda\in\Lambda\}$ so that the bijection $X_\lambda\mapsto
\widetilde{X}_\lambda$ from $\Irr_{ex}(B)$ to
$\Irr_{ex}(\widetilde{B})$ satisfies
$(\widetilde{X}_{\lambda_1}-\widetilde{X}_{\lambda_2})^G=\delta(X_{\lambda_1}-X_{\lambda_2})$
for some fixed $\delta\in\{\pm1\}$ and every
$\lambda_1,\lambda_2\in\Lambda$. As mentioned in the proof of
\cite[Theorem~3.4]{Navarro04}, Dade's bijection commutes with the
action of $\mathcal{H}_B$, and hence preserves the $p$-rationality
level. This allows us, for the purpose of proving the claim, to
assume that $\widetilde{D}\vartriangleleft G$.

Let $\widetilde{C}:=\bC_G(\widetilde{D})$ and
$\widetilde{b}:=(b_0)^{\widetilde{C}}$. By \cite[Section~4]{Dade66},
the exceptional characters of $B$ are induced from (nontrivial)
ordinary irreducible characters of $\widetilde{b}$. In fact,
$\Irr(\widetilde{b})$ consists of $|D|$ characters
$\{\chi_\lambda\mid \lambda\in\Irr(D)\}$ and
$(\chi_{\lambda_1})^G=(\chi_{\lambda_2})^G$ if and only if
$\lambda_1=\lambda_2^z$ for some $z\in E$, so that
\[\Irr_{ex}(B)=\{(\chi_\lambda)^G\mid \lambda\in\Lambda\} \text{ and }
X_\lambda=(\chi_\lambda)^G.\]

It was shown in \cite[p.\,1136]{Navarro04}, using the
character-valued formula of $\chi_\lambda$ in
\cite[Lemma~3.2]{Dade66}, that a Galois automorphism
$\tau\in\mathcal{H}_B$ moves the character $\chi_\lambda$ in
$\Irr(\widetilde{b})$ to the character in
$\Irr((\widetilde{b})^\tau)$ labeled by $\chi_{\lambda^\tau}$. Recall the groups 
$\mathcal{I}:=\mathrm{Gal}(\QQ_{|G|}/\QQ_{|G|_{p'}})\leq \mathcal{H}_B$ and the $p$-subgroup $\mathcal{I}'\leq \mathcal{I}$ from Section \ref{sec:Galauts}.
Note that every relevant block is point-wisely fixed by
$\mathcal{I}$. It follows that, for every $\tau\in\mathcal{I}$, $
(\chi_\lambda)^\tau=\chi_{\lambda^\tau}, $ which implies that
\[
\lev(\chi_\lambda)=\lev(\lambda).
\]

Further, recall that to show $\lev(X_\lambda)=\lev(\lambda)$, it suffices to show they are stable under the same elements of $\mathcal{I}'$, assuming that $\lev(X_\lambda)\geq 1$ if $p\neq 2$.

 For each $\tau\in
\mathcal{I}'$, we have
\[
X_\lambda^\tau=\left((\chi_\lambda)^G\right)^\tau=\left((\chi_\lambda)^\tau\right)^G=(\chi_{\lambda^\tau})^G.
\]
Therefore, $(\chi_\lambda)^G$ is $\tau$-invariant if and only if
$(\chi_\lambda)^G=(\chi_{\lambda^\tau})^G$, which is equivalent to
$\lambda^\tau=\lambda^z$ for some $z\in E$. We may assume
that $z$ has $p'$-order, because $E/C$ has $p'$-order and $C$ fixes
every irreducible character of $D$. Further, note that the actions of $z$ and of $\tau$ commute. 
Let $t:=|\tau|$ be the order of $\tau$, which is a
$p$-power. Then $\lambda=\lambda^{\tau^t}=\lambda^{z^t}$. Thus
$z^t\in C$ as $E/C$ acts Frobeniusly on $\Irr(D)$. But $|z|$ and
$t$ are coprime,  so $z\in C$ and $\lambda^\tau=\lambda$. We
indeed have shown that, for every $\lambda\in\Lambda$,
\[
\text{if } p=2, \text{ then } \lev(X_\lambda)=\lev(\lambda)
\]
and
\[
\text{if } p>2 \text{ and } \lev(X_\lambda)\geq1, \text{ then }
\lev(X_\lambda)=\lev(\lambda).
\]
The proof of the claim is completed.

Note that the desired equality $\lev(\chi_{\bN_G(D)})=\lev(\chi)$ is
obvious when $\chi$ is $p$-rational. By the above claim, it suffices
to prove the equality for those characters $X_\lambda$ with
$\lev(X_\lambda)=\lev(\lambda)\geq 1$. For convenience, let
$\ell:=\lev(\lambda)$. Since
$\lev(X_\lambda)=\max\{\lev(X_\lambda(g)): g\in G\}$, there exists
some $g\in G$ such that
\[
\lev(X_\lambda(g))=\lev(X_\lambda)=\lev(\lambda)=\ell.
\]
Our job now is to show that such an element $g$ can be chosen to be
inside $\bN_G(D)$. In fact, we will see that, up to conjugation,
this must be the case.

Note that $X_\lambda$ takes value 0 on every element whose $p$-part
is not conjugate to an element of $D$. For our purpose of analyzing
the value $X_\lambda(g)$, we may therefore assume that $g_p\in D$. We next claim that $g_p\in D_0-D_1$, so that $g_p$ generates $D$.

Assume, to the contrary, that $g_p\in D_1$; that is, $|g_p|\leq
p^{a-1}$. Then there exists $1\leq i\leq a-1$ such that
$|g_p|=p^{a-i}$ and $g_p\in D_i-D_{i+1}$. Now $(g_p)^h\in
D_i-D_{i+1}$ for every $h\in N_i$. We have
\[\lev(\lambda^h(g_p))=\lev(\lambda((g_p)^{h^{-1}}))=
\begin{cases} \ell-i &\text{ if } i\leq \ell,\\ 0 &\text{ if }
i>\ell
\end{cases}\] for every $h\in N_i$. In any case,
\[
\lev(\lambda^h(g_p))\leq \ell-1.
\]
By the character-valued formula (\ref{eq:2}),
\[X_\lambda(g)=\frac{\delta}{|C_i|} \sum_{h\in N_i}
\lambda^h(g_p)(\varphi_i)^h(g_{p'})\] for some $\delta\in\{\pm 1\}$.
Note that each value $(\varphi_i)^h(g_{p'})$ is $p$-rational. It
follows that $X_{\lambda}(g)$, being a sum of complex numbers of
level at most $\ell-1$, must have level at most $\ell-1$, which is a
contradiction.

We have shown that $g_p\in D_0-D_1$. In other words, $g_p$ is a generator for $D$. Thus \[g_{p'}\in
\bC_G(g_p)=\bC_G(D) \subseteq \bN_G(D),\] and it follows that
\[g=g_pg_{p'}\in \bN_G(D).\] Let $\chi:=X_\lambda$. We then have
\[\ell=\lev(\chi(g))\leq \lev(\chi_{\bN_G(D)})\leq\lev(\chi)=\ell,\]
implying that $\lev(\chi_{\bN_G(D)})=\lev(\chi)$. The proof is
complete.
\end{proof}

\begin{remark} Our proof of Theorem~\ref{mainthm:cyclic-defect-case} indeed
shows that, in view of Remark~\ref{remark-1}, if a group element
captures the $p$-rationality of the character, then that element
\emph{must} lie inside the defect normalizer, up to conjugation. At the moment, we do not know if this is true for arbitrary defect.\end{remark}

We conclude this section with the confirmation of
Conjecture~\ref{conj:main} for all characters in blocks of defect at
most 2.

\begin{proposition}
Let $p$ be a prime, $G$ a finite group, and $\chi$ be a height-zero
character in a block $B$ of $G$ with $\lev(\chi)\geq 2$ and
$d(B)\leq 2$. Suppose that $D$ is a defect group of $B$. Then
$\lev(\chi)=\lev(\chi_{\bN_G(D)})$.
\end{proposition}

\begin{proof}
If $|D|\leq p^2$ then either $D$ is cyclic or $\exp(D)\leq p$. The
former case is solved by Theorem~\ref{thm:cyclic-defect-case}. In
the latter case, $\lev(\chi)\leq 1$ by Lemma~\ref{lem:1} and the
result is trivial.
\end{proof}


\section{Imprimitive characters of prime degree}\label{sec:imprimitive}

This section proves Theorem~\ref{mainthm:prime-degree-case} in the
case where the character in question is imprimitive. Recall that a
character $\chi\in\Irr(G)$ is termed \emph{imprimitive} if there exists
a subgroup $H<G$ and $\psi\in\Irr(H)$ such that $\chi=\psi^G$.

We shall need a $p$-local invariant of characters that was
introduced recently in Isaacs-Navarro's solution
\cite{Isaacs-Navarro22} of Conjecture~\ref{conj:Navarro-Tiep1}
for $p$-solvable groups.

\begin{definition} For a character $\Psi$ (not necessarily irreducible) of a
finite group $G$ and a nonnegative integer $i$, let
\[
\Delta_i(\Psi):=\sum_{\substack{\chi\in\Irr(G)\\ \lev(\chi)=i}}
[\chi,\Psi] \chi
\]
and, if one of $\Delta_i(\Psi)(1)$ is not divisible by $p$,
\[
\ell(\Psi):=\max\{i\in\ZZ_{\geq 0}: \Delta_i(\Psi)(1) \notequiv 0
\bmod p\}.
\]
\end{definition}

\begin{lemma}\label{lem:2} Let $\Psi$ be a character of a finite group $G$ with $\lev(\Psi)=a$. We
have:
\begin{enumerate}[\rm(i)]

\item $\Delta_i(\Psi)(1)\equiv 0 \bmod p$ for every $i\geq \max\{2,a+1\}$.

\item If $a\geq 1$, then $\lev(\Psi)\geq \ell(\Psi)$.
\end{enumerate}
\end{lemma}

\begin{proof}
Note that Part (ii) follows from (i), so it is
sufficient to prove (i).

Clearly $|G|_p\geq p^a$. Note that if $p=2$ then $a=0$ or is at
least $2$. Let $\mathcal{J}$ denote the (cyclic) $p$-group
$\mathrm{Gal}(\QQ_{|G|}/\QQ_{p^a|G|_{p'}})$ if $a>0$ or simply the $p$-group
$\mathrm{Gal}(\QQ_{|G|}/\QQ_{p|G|_{p'}})$ if $a=0$. Then $\Psi$ is
$\mathcal{J}$-invariant. Since $\QQ(\psi^\tau)=\QQ(\psi)$ for every
character $\psi$ of $G$ and every $\tau\in\mathcal{J}$, each
$\Delta_i(\Psi)$ is $\mathcal{J}$-invariant, and hence $\mathcal{J}$
permutes the irreducible constituents of $\Delta_i(\Psi)$. Let $i>a$
if $a>0$ or $i\geq 2$ if $a=0$. Then each constituent of
$\Delta_i(\Psi)$, of level $i$, is not $\mathcal{J}$-invariant, and
thus belongs to a $\mathcal{J}$-orbit of nontrivial length, which is
necessarily a nontrivial $p$-power. As the irreducible constituents
of $\Delta_i(\Psi)$ is a disjoint union of these orbits, the
statement follows.
\end{proof}

The next result makes use of some ideas in the proof of
\cite[Theorem~3.5]{Isaacs-Navarro22}.

\begin{lemma}\label{lem:3}
Let  $P\leq K\leq G$ where $P\in \Syl_p(G)$, $\chi\in\Irr_{p'}(G)$,
and $\psi\in\Irr(K)$ such that $\chi=\psi^G$. Let $i\in\ZZ_{\geq
2}$. Then
\[
\Delta_i(\chi_P)(1)\notequiv 0 \bmod p \text { if and only if }
\Delta_i(\psi_P)(1)\notequiv 0 \bmod p.
\]
In particular, if $\max\{\ell(\chi_P),\ell(\psi_P)\}\geq 2$ then
$\ell(\chi_P)=\ell(\psi_P)$.
\end{lemma}

\begin{proof} Let $X$ be a set of representatives for the double
$K-P$ cosets in $G$, so that
\[
G=\bigcup_{x\in X} KxP
\]
is a disjoint union. We decompose
\[
X=X_1\cup X_2,
\]
where $X_1$ consists of those $x\in X$ such that $P\subseteq K^x$
and $X_2$ is, of course, the complement of $X_1$ in $X$. Using
Mackey's theorem (see \cite[Problem 5.6]{Isaacs1}), we have
\begin{align*}
\chi_P=\sum_{x\in X} \left((\psi^x)_{K^x\cap P}\right)^P&=
\sum_{x\in X_1} \left((\psi^x)_{K^x\cap P}\right)^P + \sum_{x\in
X_2} \left((\psi^x)_{K^x\cap P}\right)^P\\
&= \sum_{x\in X_1} (\psi^x)_P + \sum_{x\in X_2}
\left((\psi^x)_{K^x\cap P}\right)^P.
\end{align*}
Therefore, for every $i\in\ZZ_{\geq 0}$,
\[
\Delta_i(\chi_P)=\sum_{x\in X_1} \Delta_i((\psi^x)_P) + \sum_{x\in
X_2} \Delta_i(\left((\psi^x)_{K^x\cap P}\right)^P)
\]

For each $x\in X_2$, note that $K^x\cap P$ is a proper subgroup of
$P$, and it follows from \cite[Lemma 3.1]{Isaacs-Navarro22} that
\[
\Delta_i\left(\left((\psi^x)_{K^x\cap P}\right)^P\right)(1) \equiv 0
\bmod p,
\]
for every $i\geq 2$. We now obtain
\[
\Delta_i(\chi_P)(1) \equiv \sum_{x\in X_1} \Delta_i((\psi^x)_P)(1)
\bmod p.
\]

Let $x\in X_1$. Note that $\alpha$ is an irreducible constituent of
$\psi_P$ if and only in $\alpha^x$ is an irreducible constituent of
$(\psi^x)_{P^x}$, and $\lev(\alpha)=\lev(\alpha^x)$. Thus, there is
a natural bijection between irreducible constituents of $\psi_P$ and
$(\psi^x)_P$ preserving the $p$-rationality level. In particular,
\[
\Delta_i(\psi_P)(1)=\Delta_i((\psi^x)_P)
\]
for every $x\in X_1$. The last congruence in the previous paragraph
then yields
\[
\Delta_i(\chi_P)(1) \equiv |X_1|\cdot\Delta_i(\psi_P)(1) \bmod p.
\]
Now, by \cite[Lemma 3.4]{Isaacs-Navarro22}, which states that
$|X_1|$ is not divisible by $p$, the lemma follows.
\end{proof}

We can now prove Theorem~\ref{mainthm:prime-degree-case} in the
imprimitivity case.

\begin{theorem}\label{imprimitive}
Let $p$ be a prime and $G$ a finite group. Let $\chi=\lambda^G\in
\Irr(G)$ for some linear character $\lambda$ of a subgroup $K$ of
$G$ of prime index not equal to $p$. Suppose $\lev(\chi)\geq 1$.
Then $\QQ_{p^{\lev(\chi)}}=\QQ_p(\chi_P)$.
\end{theorem}

\begin{proof}
Since the conclusion is obvious when $\lev(\chi)=1$, we assume that
$a:=\lev(\chi)\geq 2$. By the character-induction formula, we have
\[
\lev(\lambda)\geq \lev(\chi)=a.
\]
On the other hand, since $\lambda$ is linear, we have $\lev(\lambda_P)=\ell(\lambda_P)$ and
\[
\lev(\lambda)=\ell(\lambda_P)=\nu(\ord(\lambda)).
\]
Furthermore, using Lemmas~\ref{lem:2} and \ref{lem:3}, we obtain
\[
\lev(\chi_P)\geq \ell(\chi_P)
\]
and
\[
\ell(\chi_P)=\ell(\lambda_P).
\]
The displayed (in)equalities imply that
$a=\lev(\chi)=\lev(\chi_P)=\ell(\chi_P)$.

Note that $\QQ_p(\chi_P)\subseteq \QQ_{p^a}$ (see \cite[Lemma
7.1]{Navarro-Tiep21}). Let $\tau\in \Gal(\QQ_{p^a}/\QQ_p(\chi_P))$.
Since $[\QQ_{p^a}:\QQ_p]=p^{a-1}$, we have that $\tau$ has $p$-power
order. Also, $\tau$ fixes $\chi_P$, and hence $\tau$ permutes the
linear constituents of $\chi_P$ of level $a$. As $\ell(\chi_P)=a$,
it follows that $\Delta_a(\chi_P)(1)\notequiv 0 \bmod p$, which
implies that the number of the linear constituents of $\chi_P$ of
level $a$ is not divisible by $p$. Therefore one of them must be
$\tau$-invariant, and therefore $\tau$ fixes $\QQ_{p^a}$ or, in
other words, $\tau$ is trivial. We have shown that
$\QQ_{p^a}=\QQ_p(\chi_P)$, as desired.
\end{proof}


\section{Quasisimple groups}\label{sec:quasisimple}

\subsection{Theorem~\ref{mainthm:prime-degree-case} for quasisimple groups}
The main result of this section is the following, which proves
Theorem~\ref{mainthm:prime-degree-case} for quasisimple groups, when
combined with Lemma~\ref{lem:observation} below and
\cite[Theorem~A3]{Navarro-Tiep21}. This result will be used in Section~\ref{sec:primitive} to prove Theorem~\ref{mainthm:prime-degree-case} for primitive characters.

\begin{theorem}\label{thm:quasisimple}
Let $M$ be a quasisimple group and $\chi\in\Irr(M)$ be of prime
degree with $\lev(\chi)\geq 2$. Let $p$ be a prime not equal to
$\chi(1)$. Then
\[\lev(\chi)=\lev(\chi_P),\] where $P\in\Syl_p(M)$.
\end{theorem}

Our next result reduces us to the case that $G$ is a group of Lie type defined in characteristic distinct from $p$. 

\begin{theorem}\label{thm:nonLiecross}
Let $p$ be a prime and let $G$ be a quasisimple group such that $S=G/\bZ(G)$
is an alternating group, a sporadic simple group, a simple group of
Lie type with exceptional Schur multiplier, or a simple group of Lie
type defined in characteristic $p$. Let $\chi\in\irr{G}$ have height zero,
lie in a block with nontrivial defect group $D$, and be such that
$\lev(\chi)\geq 2$. Then $\lev(\chi)=\lev(\chi_D)$. In particular,
Conjecture~\ref{conj:main} and Theorem \ref{thm:quasisimple} hold in
these cases.
\end{theorem}

\begin{proof}

When $S$ is either a sporadic simple group or $\Al_n$ with $5\leq
n\leq 7$, we have $\lev(\chi)\leq 1$ for all primes $p$ and all $\chi\in\irr{G}$ of height zero, which can be readily checked
in \cite{GAP}. If $S$ is a group of Lie type with exceptional Schur multiplier or
the Tits group $\tw{2}\type{F}_4(2)'$, we see using \cite{GAP} that
$\lev(\chi)\leq 1$ for all height-zero characters of $G$ except when
$p=2$ and $S=\PSL_3(4)$ with $4\mid|\zent{G}|$;
 $S=\type{B}_3(3)$ with $2\mid |\zent{G}|$;
$S=\PSU_4(3)$ with $4\mid|\zent{G}|$; or
 $S=\tw{2}\type{F}_4(2)'$.
In the latter cases, $\lev(\chi)\leq 2$, and we in fact see using \cite{GAP} that for every $\chi\in\Irr(G)$ and for every prime
$p\mid|G|$, we have $\lev(\chi)=\lev(\chi_D)$.

If $p=2$ and $S$ is $\Al_n$, then every irreducible
character of $S$ is $p$-rational (see \cite[\S3]{Hung-Tiep23} for
instance), and we are done. Now  consider the case $p=2$ and $S$ is a simple group of Lie type defined in characteristic $2$ with nonexceptional Schur multiplier. Note that by \cite{Hum}, the blocks with positive defect are in fact of maximal defect. That is, the nontrivial defect groups are Sylow $2$-subgroups in this case. Then we have $\lev(\chi)=\lev(\chi_D)$ by \cite[Theorem~A3]{Navarro-Tiep21}.

Finally, suppose that $p$ is odd and $G$ is a cover of an alternating
group $\Al_n$ with $n\geq 8$ or a quasisimple group of Lie type that
is a quotient of $\bG^F$ for some simple, simply connected algebraic
group $\bG$ over a field of characteristic $p$ and a Steinberg
endomorphism $F:\bG\rightarrow \bG$.  It was shown in
\cite[Theorem~6.1]{Navarro-Tiep21} that, in this situation, 
\[
\QQ(\chi)\subseteq \QQ_{|G|_{p'}}(\sqrt{p})
\]
for every $\chi\in\Irr(G)$. As $\QQ_{|G|_{p'}}$ contains a primitive $4$th root of unity and the conductor of 
$\sqrt{(-1)^{(p-1)/2}p}$ is $p$, it follows that $c(\chi)$ divides
$p|G|_{p'}$. Then $\lev(\chi)\leq 1$ and the conjecture trivially
holds in this case.
\end{proof}

We are now ready to complete the proof of Theorem \ref{thm:quasisimple}.

\begin{proof}[Proof of Theorem \ref{thm:quasisimple}]
By Theorem \ref{thm:nonLiecross}, we may assume that $S=M/\zent{M}$ is a simple group of Lie type defined in characteristic $q_0\neq p$, and that $S$ has non-exceptional Schur multiplier. Further, we assume that $p$ is odd, since   the statement follows from \cite[Theorem~A3]{Navarro-Tiep21} if $p=2$.

Now, \cite[Theorem~4.2]{HTZ} gives a list of the possible $(S, r)$ in
this case, where $r=\chi(1)$ is the prime for which $M$ has an
irreducible character of degree $r$. By our assumptions, we are not
in the cases listed in (i) or (vi) of \cite[Theorem~4.2]{HTZ}. Then  $S$ is one of: \[\PSL_2(q), \PSU_n(q) (n\geq 3), \PSL_n(q)
(n\geq 3), \text{ or } \PSp_{2n}(q),\] with specifications on $n, q,
r$, and $\chi$ in each case. Here, we have  $S=G/\zent{G}$ and
$M=G/Z$ for some $Z\leq \zent{G}$ and $G:=\bG^F$, where $\bG$ is a
simple, simply connected algebraic group and $F\colon
\bG\rightarrow\bG$ is a Frobenius endomorphism defining $\bG$ over
$\FF_q$, where $q$ is some power of $q_0$. We discuss each case separately.

\medskip

(I) First, suppose we are in case (ii) of \cite[Theorem~4.2]{HTZ}, so $S=\PSL_2(q)$. Then either $r=q=q_0$ and $\chi$ is the Steinberg character, which is rational; or $q$ is odd and $\chi(1)=r=\frac{q-\epsilon}{2}$ for some $\epsilon\in\{\pm1\}$; or $q$ is a power of $2$ and $r=q+\epsilon$ for some $\epsilon\in\{\pm1\}$ is a Mersenne prime or Fermat prime. In the case $q$ is odd, we have $c(\chi)\in\{q_0, 1\}$ as in \cite[\S 5.2.2]{HTZ}, so  $\lev(\chi)=0$.

So, assume $q$ is a power of $2$. Here we have  $\QQ(\chi)\subseteq \QQ_{q-\epsilon}$, so we may assume $p\mid (q-\epsilon)$. In this case, $\chi$ is the restriction of a semisimple character $\wt\chi$ of $\wt{G}=\GL_2(q)$ indexed by a semisimple element with eigenvalues $\{\zeta^i, \zeta^{-i}\}$, where $\zeta$ is  a primitive $(q-\epsilon)$ root of unity and $1\leq i<q-\epsilon$. Letting $\alpha$ be a generator for $\Irr(C_{q-\epsilon})$, we have $\QQ(\chi)\subseteq \QQ(\alpha^i)$, so that $\lev(\chi)\leq \lev(\alpha^i)$. Note that by \cite[Lemmma~4.1]{Navarro-Tiep21}, $\lev(\chi)$  is the smallest positive integer $e$ such that $\chi^{\sigma_e}=\chi$, since $p$ is odd. (Recall from Section \ref{sec:Galauts} that $\sigma_e\in\mathcal{I}'$ is the element mapping any $p$-power root of unity
$\omega$ to $\omega^{1+p^e}$.)

Now, the value of $\chi$ on a semisimple element $g_j$ of $S=G=\SL_2(q)$ with
eigenvalues $\{\zeta^j, \zeta^{-j}\}$ is $\zeta^{ij}+\zeta^{-ij}$.
In particular, taking $m:=(q-\epsilon)_{p'}$ and $h:=g_m$, we have
\[\chi(h)=\zeta^{im}+\zeta^{-im},\] which is stable under $\sigma_e$ if
and only if $\zeta^{im}$ is, since $p$ is odd. (This is worked out,
for example, as in \cite[Lemma~2.1]{penapryor}.) But this happens if
and only if $\alpha^{im}$, and hence $\alpha^i$, is stable under
$\sigma_e$. It follows that
\[\lev(\alpha^i)=\lev(\chi(h))\leq\lev(\chi).\] This establishes that
\[\lev(\alpha^i)= \lev(\chi)=\lev(\chi|_P),\] where $P$ is a Sylow
$p$-subgroup of $S$ containing $h$.

\medskip

(II) Next, suppose we are in case (iii) of \cite[Theorem~4.2]{HTZ}, so that $S=\PSL_n(q)$ with $n$ an odd prime, $q=q_0^f\geq 3$ with $q_0$ a prime and $f$ odd, and $r=(q^n-1)/(q-1)$ with $(n,q-1)=1$. Note that this means $S=M=G$. Here, as in \cite[\S5.2.1]{HTZ}, we have $\chi$ is one of the $q-2$ irreducible Weil characters of degree $r=(q^n-1)/(q-1)$. Note that $\chi$ extends to an irreducible Weil character $\wt\chi$ of $\wt{G}:=\GL_n(q)$. Further,  $\chi$ is determined by the irreducible constituent of $\wt\chi|_{\zent{\wt{G}}}$.

Let $\alpha$ be a generator for $\Irr(\zent{\wt{G}})\cong C_{q-1}$.
Then for $i=1,\ldots, q-2$,  let $\wt\chi_i$ be an irreducible Weil
character of $\wt{G}$ such that
$\wt\chi_i|_{\zent{\wt{G}}}=\wt\chi_i(1)\alpha^i$. Then any such
choice of $\wt\chi_i$ has the same restriction to $G$, and we let
$\chi_i:=\wt{\chi}_i|_{G}$ be that restriction. Let
$\wt\chi:=\wt\chi_i$ and $\chi:=\chi_i$. As discussed in
\cite[\S5.2.1]{HTZ}, we have $\QQ(\wt\chi_i)\subseteq
\QQ(\alpha^i)$, so that \[\QQ(\chi)\subseteq \QQ(\wt\chi)\subseteq
\QQ(\alpha^i)= \QQ(\zeta^i)\subseteq \QQ(\zeta),\] where $\zeta$ is
a primitive $(q-1)$-root of unity in $\overline{\QQ}^\times$. From
this, we may assume $p\mid (q-1)$, as otherwise $\lev(\chi)=0$.

Now, following \cite[p.~125]{GT99}, we see $\wt\chi=\hc{L}{\wt{G}}(\lambda)$, where $L$ is a Levi subgroup of the form $\GL_1(q)\times \GL_{n-1}(q)$ of $\wt{G}$, $\lambda\in\irr{L}$ is the character of a module of the form \[S(s,(1))\otimes S(t, (n-1))\] in the notation of loc. cit. with $s\neq t\in\FF_q^\times$ and $s/t=\alpha^i$, and $\hc{L}{\wt{G}}$ denotes Harish-Chandra induction. Here by an abuse of notation, we also denote by $\alpha$ a generator of $\FF_q^\times\cong C_{q-1}$. Since multiplying by a linear character of $\wt{G}$ does not affect $\chi_i$, we may further assume that $t=1$, so \[\lambda=S(\alpha^i,(1))\otimes S(1, (n-1)).\] (Note that for $\alpha^j\in\FF_q^\times$,  the module $S(\alpha^j, (k))$ affords the inflation of the linear character $\GL_k(q)/\SL_k(q)\cong \FF_q^\times\rightarrow\overline{\QQ}^\times$ defined by $\alpha\mapsto \zeta^j$.)

Now, let $Q$ be a parabolic subgroup of $\wt{G}$ such that $L\leq Q$ is the Levi complement in $Q$, and let $\hat\lambda$ be the inflation of $\lambda$ to $Q$. Then \[\chi=\Res^{\wt{G}}_G\Ind_Q^{\wt{G}}(\hat\lambda)=
\Ind_{Q\cap G}^G\Res_{Q\cap G}^Q(\hat\lambda)\] since $\wt{G}=GQ$. Note that $r=\chi(1)=[G:Q\cap G]$. Then by the first half of the proof of Theorem \ref{imprimitive}, we have $\lev(\chi)=\lev(\chi_P)$, as desired.

\medskip

(III) Now consider case (iv) of \cite[Theorem~4.2]{HTZ}. Here $S=\PSU_n(q)$ with $n$ an odd prime,  $r=(q^n+1)/(q+1)$, and   $(n,q+1)=1$. Again, this means $S=M=G$ and  $\chi$ is one of the $q$ irreducible Weil characters of degree $r=(q^n+1)/(q+1)$ (see  \cite[\S5.3]{HTZ}).
Again,  $\chi$ extends to an irreducible Weil character $\wt\chi$ of $\wt{G}:=\GU_n(q)$ and  $\chi$ is determined by its values on $\zent{\wt{G}}$.  Letting $\alpha$ be a generator for $\Irr(\zent{\wt{G}})\cong C_{q+1}$, we have Weil characters $\chi_i=\wt\chi_i|_{G}$, where $\wt\chi_i|_{\zent{\wt{G}}}=\wt\chi_i(1)\alpha^i$ as in (II), now for $i=1,\ldots, q$.
 Here we similarly have $\QQ(\chi_i)\subseteq \QQ(\alpha^i)\subseteq \QQ(\xi)$, where $\xi$ is a primitive $(q+1)$-root of unity in $\overline{\QQ}^\times$. So, we  assume $p\mid (q+1)$. Further, this establishes that \[\lev(\chi_i)\leq \lev(\alpha^i).\] 

In this case, we use the explicit formula from \cite[Theorem~4.1]{TZ97} for the values of $\chi_i$. Namely, for $g\in \wt{G}$, we have
\[\chi_i(g)=\frac{(-1)^n}{q+1}\sum_{k=0}^q\xi^{-ik}(-q)^{\dim\ker{g-\hat\xi^{-k}}},\] where $\hat\xi$ denotes a generator of the subgroup $C_{q+1}\leq \FF_{q^2}^\times$.

Now, let $m:=(q+1)_{p'}$ and let $\beta:=\hat\xi^{m}$ be a generator of the Sylow $p$-subgroup of $C_{q+1}\leq\FF_{q^2}^\times$. Let $h$ be a semisimple $p$-element of $G$ with eigenvalues $\{\beta, \beta^{-1}, 1,\ldots,1\}$. Then  since $n$ and $p$ are odd, we have:

\[\chi_i(h)=\frac{-1}{q+1}\left((\xi^{im}+\xi^{-im})(-q)+(-q)^{n-2}- \xi^{im}-\xi^{-im}+\sum_{k=1}^q \xi^{-ik} \right)\]
\[=(\xi^{im}+\xi^{-im})+\frac{1}{q+1}\left(q^{n-2}-\sum_{k=1}^q \xi^{-ik} \right)=(\xi^{im}+\xi^{-im})+\frac{q^{n-2}+1}{q+1}.\]
 Recall again that by \cite[Lemma~4.1]{Navarro-Tiep21}, $\lev(\chi)$  is the smallest positive integer $e$ such that $\chi^{\sigma_e}=\chi$, since $p$ is odd.
From above, we see $\chi_i(h)$ is fixed by $\sigma_e$ if and only if $(\xi^{im}+\xi^{-im})$ is fixed by $\sigma_e$. From here, we conclude similar to  (I). Namely,  $\chi_i(h)$ is then stable under $\sigma_e$ if and only if $\xi^{im}$ is stable under $\sigma_e$, if and only if the character $\alpha^{im}$ is stable under $\sigma_e$, so also $\alpha^i$ is. So we have $\lev(\chi_i(h))=\lev(\alpha^i)$. Then \[\lev(\chi_i)\geq \lev(\alpha^i),\] forcing that
\[\lev(\alpha^i)= \lev(\chi_i)=\lev(\chi_i|_P),\] where $P$ is a Sylow $p$-subgroup of $S$ containing $h$.

\medskip

(IV) Finally, assume we are in case (v) of \cite[Theorem~4.2]{HTZ}. Then $S=\PSp_{2n}(q)$ and either \begin{itemize}
\item[(a)] $r=(q^n+1)/2$ with $n=2^a\geq 2$ and $q=q_0^{2^k}$ with $q_0$ odd and $k\geq 0$; or
\item[(b)] $r=(3^n-1)/2$, where $n$ is an odd prime and $q=3$.

\end{itemize}
In case (a), we have as in \cite[\S5.4.1]{HTZ} that either $q$ is a square and hence $\chi$ is rational-valued, or $k=0$ so $q=q_0$ and $\QQ(\chi)\subseteq \QQ(\zeta_{q})$, where $\zeta_{q}$ is a primitive $q$th root of unity in $\overline{\QQ}^\times$. Then $\chi$ is almost $p$-rational (hence has $\lev(\chi)\leq 1$) if $p=q_0$ and is $p$-rational (so $\lev(\chi)=0$) if $p\neq q_0$.
In case (b), we similarly have $\QQ(\chi)\subseteq \QQ(\zeta_3)$ with $\zeta_{3}$  a primitive $3$rd root of unity. Again, $\lev(\chi)\leq 1$ if $p=3$ and $\lev(\chi)=0$ if $p\neq 3$.
\end{proof}


\subsection{Further examples satisfying Conjecture \ref{conj:main}}\label{sec:further-examples}

We provide further evidence
for Conjecture~\ref{conj:main} among certain (almost)-quasisimple groups.

When $G=\bG^F$, where $\bG$ is a connected reductive algebraic group over $\overline{\mathbb{F}}_{q_0}$ for a prime $q_0$ and $F\colon \bG\rightarrow\bG$ is a Steinberg morphism, the set $\Irr(G)$ is partitioned into so-called rational Lusztig series $\mathcal{E}(G, s)$. Here,  $s$ ranges over semisimple elements, up to $G^\ast$-conjugacy, of $G^\ast=(\bG^\ast)^{F^\ast}$, where $(\bG^\ast, F^\ast)$ is dual to $(\bG, F)$.  Let $p\neq q_0$ be a nondefining prime and suppose that $s\in G^\ast$ is a semisimple element of order relatively prime to $p$.
Then we define $\mathcal{E}_p(G,s)$ to be the union of all $\mathcal{E}(G, st)$
where $t\in \cent{G^\ast}{s}$ is a $p$-element.  A result of Digne and Michel yields that the set $\mathcal{E}_p(G, s)$ is a union of $p$-blocks. (See \cite[Theorem~9.12]{CE04}.)

Our next examples concerning Conjecture \ref{conj:main} are the Suzuki and Ree groups. That is, these are the cases that $F$ is not a Frobenius morphism.

\begin{theorem}\label{thm:suzree}
 Conjecture \ref{conj:main}  holds for $\tw{2}\type{B}_2(q^2)$ for $q^2=2^{2n+1}>2$ and $\tw{2}\type{G}_2(q^2)$ for $q^2=3^{2n+1}>3$. Further, if Conjecture \ref{conj:Navarro-Tiep1} holds for $\tw{2}\type{F}_4(q^2)$, where $q^2=2^{2n+1}>2$, then Conjecture \ref{conj:main} holds for $\tw{2}\type{F}_4(q^2)$ for $p\neq 3$.
\end{theorem}
\begin{proof}
By Theorems \ref{thm:nonLiecross} and \ref{thm:cyclic-defect-case}, we may assume that $p$ is not the defining characteristic for $G$ and that the Sylow $p$-subgroups of $G$ are non-cyclic.
Then we are left to consider the case that $G=\tw{2}\type{F}_4(q^2)$ with $q^2=2^{2n+1}$ and $p$ is an odd prime dividing $ (q^2-1)$, $(q^2+1)$, or $(q^4+1)$. 

If $p\nmid(q^2-1)$  \cite[Bemerkung~1]{malle91} yields that each $\mathcal{E}_p(G, s)$ for $s\in G^\ast$ a semisimple $p'$-element contains a unique block of positive defect, which therefore has as defect groups a Sylow $p$-subgroup of $\cent{G^\ast}{s}$, using \cite[Lemma~2.6]{Kessar-Malle13}. If instead $p\mid(q^2-1)$, we have by \cite[Bemerkung~1]{malle91} that each $\mathcal{E}_p(G, s)$ has one or three blocks of positive defect, but only one of these is noncyclic. In either case, a block $B$ with non-cyclic defect groups has maximal defect for $p\neq 3$, completing the proof by our assumption that Conjecture \ref{conj:Navarro-Tiep1} holds. 
\end{proof}

\medskip

Our next several  examples will come from linear groups, especially in the case $p=2$.  Let $\wt{G}:=\GL_n( q)$, $G=\SL_n(q)$, and $S=\PSL_n(q)$, and assume that $q$ is odd and $p=2$. 

In this situation, consider a block $\wt{B}$ of $\wt{G}$ covering a block $B$ of $G$. These can be chosen so that $\wt{B}= \mathcal{E}_2(\wt{G}, \wt{s})$ for some odd-order semisimple element $\wt{s}$ of $\wt{G}^\ast\cong \wt{G}$, by \cite[Theorems~9.12 and 21.14]{CE04}. Further, a Sylow $2$-subgroup of $\cent{\wt{G}^\ast}{s}$ gives a defect group for $\wt{B}$ by \cite[Corollary~(5E)]{FS82}.
Now, $\cent{\wt{G}^\ast}{\wt s}$ is a product $\cent{\wt{G}^\ast}{\wt s}=\prod \GL_{m_i}(q^{d_i})$, where $m_i$ and $d_i$ correspond to the multiplicities and degrees of the eigenvalues of $\wt{s}$ and $n=\sum m_id_i$.

\begin{theorem}\label{thm:PSL2}
Let $M$ be a quasisimple group with $M/\zent{M}=\PSL_2(q)$, where
$q\geq 5$ is a prime power. Then Conjecture \ref{conj:main} holds for $M$.
\end{theorem}
\begin{proof}
Let $q$ be a power of a prime $q_0$.
 In this case, the Sylow $p$-subgroups of $M$ are cyclic unless $p\in\{2,q_0\}$. So, we may assume by Theorems \ref{thm:cyclic-defect-case} and \ref{thm:nonLiecross} that $p=2$ and by \cite[Theorem~A3]{Navarro-Tiep21} we need only consider blocks of $M$ with non-maximal defect. We have $M\in\{G, S\}$, where $G=\SL_2(q)$ and $S=\PSL_2(q)$. Let $\wt{G}=\GL_2(q)$ and let $\bar B$ be a block of $M$ with positive, non-maximal defect dominated by a block $B$ of  $G$. Let $\wt{B}$ be a block of $\wt{G}$ covering $B$. Then as discussed above, there is some odd-order, semisimple element $\wt s$ of $\wt{G}^\ast\cong \wt G$ such that $\wt{B}=\mathcal{E}_2(\wt{G}, \wt s)$
 and a Sylow $2$-subgroup of $\cent{\wt{G}^\ast}{\wt s}$ gives a defect group for $\wt{B}$.
 Then with our assumption that $\bar B$ is not of maximal defect, we see that any such defect group has cyclic intersection with $G$, so that $B$ (hence $\bar B$) has cyclic defect groups. Then we again apply Theorem \ref{thm:cyclic-defect-case}.
\end{proof}

\begin{proposition}\label{prop:typeAex}
Let $\wt{G}=\GL_n(q)$ with $q$ odd and let $\wt{B}=\mathcal{E}_2(\wt{G}, \wt s)$ be a $2$-block of $\wt{G}$ as discussed above. Further suppose that $q\equiv -1\pmod 4$ and that each $d_i$ is odd, in the notation above. Then $\wt{B}$ satisfies Conjecture \ref{conj:main}.
\end{proposition}
\begin{proof}
Let $D$ be a defect group for $\wt{B}$, such that $D=\prod D_i$ with 
$D_i\in \mathrm{Syl}_2(\GL_{m_i}(q^{d_i}))$. Note that in this case, 
$q^{d_i}\equiv -1\pmod 4$ for each factor $\GL_{m_i}(q^{d_i})$ of $C:=\cent{\wt{G}^\ast}{\wt s}$. 

We claim that in this situation,  each $D_i/D_i'$ has exponent at most $2$, and hence so does $D/D'$. We  can see this from the description of Sylow $2$-subgroups of $\GL_{m}(q^d)$ in \cite{carterfong}. Indeed,  by the description in loc. cit., such a group is a direct product of Sylow $2$ subgroups of $\GL_{2^j}(q)$ for various powers  $2^j$ of $2$. Hence, it suffices to prove the claim for $m$ a power of $2$. A Sylow $2$-subgroup $P_{2^j}$ of $\GL_{2^j}(q^d)$ is an iterated wreath product $P_2\wr C_2\wr C_2\cdots\wr C_2$, where $P_2\in\Syl_2(\GL_2(q^d))$. Since the latter is  semidihedral, we know $\exp(P_2/P_2')\leq 2$. For $j\geq 2$, we have $P_{2^j}=P_{2^{j-1}}\wr C_2$. Then  $P_{2^j}/P_{2^{j}}'\cong P_{2^{j-1}}^2/\langle (P_{2^{j-1}}')^2, [P_{2^{j-1}}^2, C_2]\rangle \times C_2$, and we can see inductively that $\exp(P_{2^j}/P_{2^j}')\leq 2$. This proves our claim. (See also \cite[Proposition~4.3]{ILNT}).

 Hence, by \cite[Theorem D]{ILNT},  each odd-degree character in the principal block $B_0(C)$ of $C$ is $2$-rational. But then the same is true for the height-zero characters of $\wt{B}$ by the main Theorem of  \cite{srinivasanvinroot}, since these correspond to the height-zero characters in $B_0(C)$ by Jordan decomposition using  \cite[Theorem~(7A)]{FS82}. 
\end{proof}

\begin{remark}
We remark that the same proof shows that when $\wt{G}=\GU_n(q)$ with $q\equiv 1\pmod 4$ and $\wt{B}=\mathcal{E}_2(\wt{G}, \wt{s})$ is a $2$-block of $\wt{G}$ with $\cent{\wt{G}^\ast}{\wt{s}}=\prod \GL^\eta_{m_i}(q^{d_i})$ with $q^{d_i}\equiv -\eta\pmod 4$ for each $i$, then Conjecture \ref{conj:main} holds for $\wt{B}$. 

Further, using the description of Sylow $2$-subgroups of $\Sp_{2n}(q)$, $\SO_{2n+1}(q)$, $O_{2n+1}(q)$, $\SO_{2n}^\pm(q)$, and $O_{2n}^\pm(q)$ in \cite{carterfong} and arguing similarly to before, we see that each of these groups also satisfy $\exp(P/P')\leq 2$ for a Sylow $2$ subgroup $P$. So, taking $G$ to be a classical-type group $\operatorname{CSp}_{2n}(q)$, $\SO_{2n+1}(q)$, or $\operatorname{CSO}_{2n}^\pm(q)$ and a block $B$ such that $B=\mathcal{E}_2(G, s)$ (again applying \cite[Theorems~9.12 and 21.14]{CE04}), we may obtain analogous examples using \cite{FS89} in place of \cite{FS82}. This could be further extended to classical types whose center is disconnected in the cases that each $\chi\in\Irr(B)$ lies in a series $\mathcal{E}(G, st)$ where $\cent{\bG^\ast}{st}$ is connected, using  \cite{STV} in place of \cite{srinivasanvinroot}.
\end{remark}


\section{Primitive characters of prime degree}\label{sec:primitive}

This section handles the primitivity case of
Theorem~\ref{mainthm:prime-degree-case}.
We begin with an easy observation.

\begin{lemma}\label{lem:observation}
Let $p$ be a prime, $G$ a finite group, $P\in\Syl_p(G)$, and
$\chi\in \Irr_{p'}(G)$ with $\lev(\chi)\geq 2$. To prove
Conjecture~\ref{conj:Navarro-Tiep1}, it suffices to show that
$\lev(\chi)=\lev(\chi_P)$ and, additionally, $\QQ_4\subseteq
\QQ(\chi_P)$ if $p=2$.
\end{lemma}

\begin{proof}
Let $a:=\lev(\chi)=\lev(\chi_P)\geq 2$. We aim to show that
$\QQ_p(\chi_P)=\QQ_{p^a}$. First we have $\QQ(\chi_P)\subseteq
\QQ_{p^a}$ (see \cite[Lemma~7.1]{Navarro-Tiep21}). If $p$ is odd
then all the subfields of $\QQ_{p^a}$ containing $\QQ_p$ are of the
form $\QQ_{p^b}$ for $1\leq b\leq a$ (see the proof of
\cite[Theorem~2.3]{Navarro-Tiep21}), and therefore the fact that
$\lev(\chi_P)=a$ forces $\QQ_p(\chi_P)$ to be the entire
$\QQ_{p^a}$.

Let $p=2$ and assume that $\QQ_4\subseteq \QQ(\chi_P)$. Now all the
subfields of $\QQ_{2^a}$ containing $\QQ_4$ are again of the form
$\QQ_{2^b}$ for $1\leq b\leq a$, and we still have
$\QQ_{2^a}=\QQ(\chi_P)$.
\end{proof}

\begin{lemma}\label{lem:faithful-red}
Let $K\nor G$ and $\overline{\chi}\in\Irr(G/K)$. Let $\chi$ be the
inflation of $\overline{\chi}$ up to $G$. Let $P\in\Syl_p(G)$ and
$\overline{P}:=PK/K\in\Syl_p(G/K)$. Then
\[
\lev(\chi)=\lev(\overline{\chi}) \text{ and }
\lev(\chi_P)=\lev(\overline{\chi}_{\overline{P}}).
\]
\end{lemma}

\begin{proof}
This follows from the fact that the sets of values of $\chi$ and
$\overline{\chi}$, as well as those of $\chi_P$ and
$\overline{\chi}_{\overline{P}}$, are the same.
\end{proof}

\begin{lemma}\label{lem:level-product}
Let $K_1,...,K_n$ be finite abelian extensions of $\QQ$. Let
$K:=K_1\cdots K_n$ denote the smallest subfield of $\CC$ containing
all $K_i$. Then $K$ is also a finite abelian extension of $\QQ$ and
\[
\lev(K)=\max\{\lev(K_i):1\leq i\leq n\}.
\]
\end{lemma}

\begin{proof}
It is easy to see that $K\subseteq \QQ_{\lcm(c(K_1),...,c(K_n))}$,
so $K$ is a finite abelian extension of $\QQ$.

Suppose $a:=\max\{\lev(K_i):1\leq i\leq n\}$. Then $K_i\subseteq
\QQ_{p^a}\QQ_{c(K_i)_{p'}}$ and thus $K\subseteq
\QQ_{p^a}\QQ_{\prod_i c(K_i)_{p'}}$, implying that $\lev(K)\leq a$.
The lemma follows as it is clear that $\lev(K_i)\leq \lev(K)$ for
every $i$.
\end{proof}

For the remainder of this section, a finite group $G$ is called
\emph{almost quasisimple} if there exists a nonabelian simple group
$S$ such that $S\nor G/\bZ(G) \leq \Aut(S)$.

We shall need the following rather technical result, extracted from
\cite{HTZ}.

\begin{lemma}\label{lem:technical}
Let $G$ be an almost quasisimple irreducible primitive subgroup of
$\GL(r,\CC)$, where $r$ is a prime, and $\chi\in\Irr(G)$ the
corresponding character. Let $S$ be the socle of $G/\bZ(G)$ and $M$
the last term in the derived series of $G$. Let $\pi$ be the set of
prime divisors of $r|\bZ(M)|$ and $Z_\pi$ denotes the Hall
$\pi$-subgroup of (the abelian group) $Z:=\bZ(G)$. Let $N:=Z_\pi M$.
Then the following hold.

\begin{enumerate}[\rm(i)]
\item $M$ is quasisimple with $M/\bZ(M)=S$ and $\chi_M$ is
irreducible.

\item  $\varphi:=\chi_N\in\Irr(N)$ and $\varphi$ has the
canonical extension $\widehat{\varphi}$ to $G$, in the sense of
\cite[Corollary~6.2]{Navarro18}. Furthermore,
$\chi=\widehat{\varphi}\lambda$ for some linear character $\lambda$
of $G/N$.

\item  $\QQ(\chi)=\QQ(\chi_M)\QQ(\mu)\QQ(\lambda)$ where
$\lambda$ is as in (ii) and $\mu$ is the irreducible (linear)
character of $Z_\pi$ lying under $\chi$.
\end{enumerate}
\end{lemma}

\begin{proof}
This follows from Proposition~3.5 and Corollary~4.4 of \cite{HTZ}
and their proofs.
\end{proof}

\begin{theorem}\label{primitive}
Let $p$ be a prime and $G$ a finite group. Let $\chi$ be an
irreducible primitive character of $G$ of prime degree not equal to $p$.
Suppose $\lev(\chi)\geq 2$. Then $\lev(\chi)=\lev(\chi_{P})$ for
$P\in\Syl_p(G)$. Furthermore, if $p=2$ then $\QQ_4\subseteq
\QQ(\chi_P)$.
\end{theorem}

\begin{proof}
We may assume that $\chi$ is faithful, by modding out by its kernel
and using Lemma~\ref{lem:faithful-red} if necessary. Let
\[r:=\chi(1).\] Then $G$ is an irreducible primitive subgroup of
$\GL(r,\CC)$. By \cite[Lemma~4.1]{HTZ}, we are in one of the
following situations.

\begin{enumerate}
\item[(A)] $G$ is almost quasisimple.

\item[(B)] $G$ contains a normal $r$-subgroup
 $R =\bZ(R)E$, where $E$ is an irreducible extraspecial $r$-group of order $r^{3}$,
 and either $R = E$ or $\bZ(R)\cong C_4$.
\end{enumerate}

\medskip

A. Consider the case when $G$ is almost quasisimple. As in
Lemma~\ref{lem:technical}, we use $Z$ for the center of $G$, $S$ for
the socle of $G/Z$, and $M$ the last term in the derived series of
$G$. Also, $\pi$ is the set of prime divisors of $r|\bZ(M)|$ and
$Z_\pi$ denotes the Hall $\pi$-subgroup of $Z$. Set $N:=Z_\pi M$.

By Lemmas~\ref{lem:level-product} and \ref{lem:technical}, we have
\[
\lev(\chi)=\max\{\lev(\chi_M),\lev(\mu),\lev(\lambda)\}
\]
for some linear characters $\mu$ of $Z_\pi$ and $\lambda$ of $G/N$.

\medskip

(i) Suppose first that $\lev(\chi)=\lev(\chi_M)$. Recall that $M$ is
quasisimple and $\chi_M\in\Irr(M)$, by Lemma~\ref{lem:technical}(i).
Using Theorem~\ref{thm:quasisimple}, we obtain
\[
\lev(\chi_M)=\lev(\chi_Q),
\]
for every $Q\in\Syl_p(M)$. Choosing such a $Q$ that is contained in
$P$, we have
\[
\lev(\chi)\geq\lev(\chi_P)\geq \lev(\chi_Q)=\lev(\chi_M)=\lev(\chi),
\]
and the first statement of the theorem follows. When $p=2$, we know
from \cite[Theorem~A3]{Navarro-Tiep21} that
$\QQ(\chi_Q)=\QQ_{2^{\lev(\chi_M)}}\supseteq \QQ_4$, and thus
$\QQ(\chi_P)\supseteq \QQ_4$, as wanted.

\medskip

(ii) Next, suppose that $\lev(\chi)=\lev(\mu)$. Let
$Q\in\Syl_p(Z_\pi)$. 
Recall that $\mu\in\Irr(Z_\pi)$ lies under $\chi$ and $Z_\pi$ is
central in $G$. Hence $\chi_{Z_\pi}$ is a rational multiple of
$\mu$, and we deduce that $\chi_Q$ is a rational multiple of $\mu_Q$
as well. Since $\lev(\chi)\geq 2$, we have $\QQ(\chi_Q)=
\QQ(\mu_Q)\supseteq \QQ_4$,  so it remains to show that
$\lev(\chi)=\lev(\chi_P)$.

Observe that
\[
\lev(\chi_Q)=\lev(\mu_Q).
\]
Clearly, $\lev(\mu)=\lev(\mu_Q)$, as $\mu$ is linear. Altogether, we
have
\[
\lev(\chi)\geq\lev(\chi_Q)=\lev(\mu_Q)=\lev(\mu)=\lev(\chi),
\]
implying that $\lev(\chi)=\lev(\chi_Q)$, and we are done again.

\medskip

(iii) Finally suppose that
$a=\lev(\chi)=\lev(\lambda)>\max\{\lev(\chi_M),\lev(\mu)\}$. Notice
that $N$ is a central product of $Z_\pi$ and $M$ and
$\varphi\in\Irr(N)$ lies over $\varphi_M\in\Irr(M)$ and
$\mu\in\Irr(Z_\pi)$. Therefore,
\[\QQ(\varphi)=\QQ(\chi_M)\QQ(\mu).\]
It follows from Lemma~\ref{lem:level-product} that
\[
\lev(\varphi)= \max\{\lev(\chi_M),\lev(\mu)\}<a.
\]
As $\QQ(\widehat{\varphi})=\QQ(\varphi)$ (see
\cite[Corollary~6.4]{Navarro18}), we then have
$\lev(\widehat{\varphi})<a$.

Again, by Lemma~\ref{lem:technical},
$\chi=\widehat{\varphi}\lambda$. So
$\chi_P=\widehat{\varphi}_P\lambda_P$. Assume to the contrary that
$\lev(\chi_P)\leq a-1$. Then, for every $g\in P$, we have
$\widehat{\varphi}(g)\neq 0$ and
\[
\lev(\lambda(g))=\lev(\chi(g)\widehat{\varphi}(g)^{-1})\leq
\max\{\lev(\chi(g)),\lev(\widehat{\varphi})\}\leq a-1,
\]
which implies that $\lev(\lambda)=\lev(\lambda_P)\leq a-1$, a
contradiction. We have proved that $\lev(\chi_P)\geq a$, which
forces $\lev(\chi)=\lev(\chi_P)$, as desired.

Now suppose that $p=2$. If either $\lev(\chi_M)\geq 2$ or
$\lev(\mu)\geq 2$ then it was already known from parts (i) and (ii)
that $\QQ_4\subseteq \QQ(\chi_P)$. So let us assume that both
$\chi_M$ and $\mu$ are $2$-rational. Both $\varphi$ and
$\widehat{\varphi}$ are then $2$-rational as well, as
$\QQ(\widehat{\varphi})=\QQ(\varphi)=\QQ(\chi_M)\QQ(\mu)$. It
follows that $\widehat{\varphi}_P$ is rational-valued. Now we see that $\QQ(\chi_P)\supseteq \QQ_4$, using
$\chi_P=\widehat{\varphi}_P\lambda_P$ together with the facts that
$\lambda$ is linear and $\lev(\lambda)\geq 2$.

\medskip

B. Next we consider the situation in which $G$ contains a normal
$r$-subgroup
 $R =\bZ(R)E$, where $E$ is an irreducible extraspecial $r$-group of order $r^{3}$,
 and either $R = E$ or $\bZ(R)\cong C_4$. By the main result of
 \cite{Isaacs-Navarro22}, we may assume that $G$ is non-solvable. Using \cite[Theorem~6.1]{HTZ},
 we obtain
\[\QQ(\chi)=\QQ(\chi_Z)=\QQ(\exp(2i\pi/a)),\] where $Z:=\bZ(G)$ and $a\in\ZZ^{+}$ is a certain
divisor of $\exp(Z)$ divisible by $r$.

Let $\theta$ be the (unique) irreducible constituent of $\chi_Z$, so
that $\chi_Z=\alpha \theta$ for some $\alpha\in\ZZ^+$. Let
$Q\in\Syl_p(Z)$. As $\theta$ is linear, we then have
\[\lev(\chi)\geq\lev(\chi_Q)=\lev(\theta_Q)=\lev(\theta)=\lev(\chi_Z)=\lev(\chi),\]
which implies that $\lev(\chi)=\lev(\chi_P)$. Also,
$\QQ(\chi_P)\supseteq \QQ(\chi_Q)=\QQ(\mu_Q)\supseteq \QQ_4$. The
proof is complete.
\end{proof}

Theorem~\ref{mainthm:prime-degree-case} now readily follows from
Theorem~\ref{imprimitive}, Theorem~\ref{primitive}, and
Lemma~\ref{lem:observation}.


\section{Further discussion}\label{sec:AMN-conjecture}

We have seen the connection between Conjecture~\ref{conj:main} and
Navarro-Tiep's Conjecture~\ref{conj:Navarro-Tiep1}. We now discuss
another connection, this time with the well-known
Alperin-McKay-Navarro (AMN) conjecture. In particular, we shall
explain how Conjecture~\ref{conj:main} implies that the conjectural
AMN bijection should respect the $p$-rationality level of the
defect-normalizer restrictions.

\subsection{Connection with the Alperin-McKay-Navarro conjecture}

Keep the notation from Section \ref{sec:Galauts}, so that $|G|=n=p^bm$
with $(p,m)=1$ and $\mathcal{G}:=\mathrm{Gal}(\QQ_n/\QQ)\cong\mathcal{I}\times \mathcal{K}$. Recall that
$\mathcal{H}=\mathcal{I}\times \langle \sigma\rangle,$ where
$\sigma\in \mathcal{K}$ is such that its restriction to $\QQ_m$ is the
Frobenius automorphism $\zeta \mapsto \zeta^p$.
 The McKay-Navarro
conjecture predicts that, for $P\in\Syl_p(G)$, there should exist an
$\mathcal{H}$-equivariant bijection from $\Irr_{p'}(G)$ to
$\Irr_{p'}(\bN_G(P))$. Such a bijection necessarily preserves the
$p$-rationality level of characters (see \cite[Section~2]{Hung22},
for instance).

As noted in Section~\ref{sec:cyclic-defect}, the Galois group
$\mathcal{G}$ permutes the $p$-blocks of $G$. Let $B$ be a $p$-block
of $G$ and $\mathcal{H}_B$ be the subgroup of $\mathcal{H}$ fixing $B$.
 Recall that $\mathcal{I}\leq \mathcal{H}_B$.

Now let $D$ be a defect group of $B$ and $b\in \Bl(\bN_G(D))$ be the
Brauer correspondent of $B$. The group $\mathcal{H}_B$ then permutes
(and preserves the height) the ordinary characters in $B$ and $b$.
The Alperin-McKay-Navarro (AMN) conjecture \cite[Conjecture
B]{Navarro04} asserts that there exists a bijection
\[^*: \Irr_0(B) \rightarrow \Irr_0(b)\] that commutes with the
action of $\mathcal{H}_B$. That is,
\[(\chi^\tau)^\ast=(\chi^*)^\tau\] for every $\chi\in \Irr_0(B)$ and
every $\tau\in \mathcal{H}_B$.

Let $\FF$ be the fixed field in $\QQ_n$ of $\mathcal{H}_B$. The AMN
conjecture then implies that
\[
\FF(\chi)=\FF(\chi^\ast)
\]
for all $\chi\in \Irr_0(B)$. As $\mathcal{I}\leq \mathcal{H}_B$, we
have $\FF \subseteq \QQ_m$, and thus $\lev(\FF)=0$. It follows that
the conjectural bijection $^\ast$ preserves the $p$-rationality
level:
\[\lev(\chi)=\lev(\chi^\ast).\]

The following is the AMN conjecture with the defect-normalizer
restriction incorporated.

\begin{conjecture}\label{conj:combine}
Let $p$ be a prime and $G$ a finite group. Let $B\in \Bl(G)$ be a
$p$-block of $G$ with defect group $D$ and $b\in \Bl(\bN_G(D))$ be
its Brauer correspondent. Let $\mathcal{H}_B$ be the subgroup of
$\mathcal{H}$ fixing $B$. Then there exists an
$\mathcal{H}_B$-equivariant bijection $^*: \Irr_0(B) \rightarrow
\Irr_0(b)$ such that $\lev(\chi_{\bN_G(D)})=\lev(\chi^\ast)$ for
every $\chi\in\Irr_0(B)$ of $p$-rationality level at least 2.
\end{conjecture}

\begin{theorem}\label{thm:relation-AMN-main}
Conjecture \ref{conj:combine} follows from the AMN conjecture and
Conjecture \ref{conj:main}. Conversely, Conjecture \ref{conj:main}
follows from Conjecture \ref{conj:combine}.
\end{theorem}

\begin{proof}
We keep the above notation. First, the AMN conjecture implies that
there exists an $\mathcal{H}_B$-equivariant bijection $^*: \Irr_0(B)
\rightarrow \Irr_0(b)$ such that $\lev(\chi)=\lev(\chi^\ast)$.
Conjecture \ref{conj:main} then implies that
$\lev(\chi_{\bN_G(D)})=\lev(\chi^\ast)$ for all $\chi\in\Irr_0(B)$
of level at least 2.

For the converse statement, let $\chi$ be a height-zero character in
a block $B$ and assume that there exists a bijection $^*: \Irr_0(B)
\rightarrow \Irr_0(b)$ that commutes with the action of
$\mathcal{H}_B$ such that $\lev(\chi_{\bN_G(D)})=\lev(\chi^\ast)$.
As mentioned, we then have $\lev(\chi)=\lev(\chi^\ast)$, and it
follows that $\lev(\chi)=\lev(\chi_{\bN_G(D)})$, as wanted.
\end{proof}

\subsection{Consequences of Conjecture~\ref{conj:main}}
A character $\chi$ is termed \emph{almost $p$-rational} if its conductor is
not divisible by $p^2$, or, equivalently, $\lev(\chi)\leq 1$ (see
\cite{Hung-Malle-Maroti21}).

\begin{consequence}
Let $\chi$ be a height-zero character of a finite group $G$ and $D$
a defect group of the $p$-block of $G$ containing $\chi$. Assume
Conjecture~\ref{conj:main} holds. Then $\chi$ is almost $p$-rational
if and only if $\chi_{\bN_G(D)}$ is almost $p$-rational. In
particular, when $p=2$, $\chi$ is $2$-rational if and only if
$\chi_{\bN_G(D)}$ is $2$-rational.
\end{consequence}

\begin{proof} It is clear that if $\chi$ is almost $p$-rational then so
is $\chi_{\bN_G(D)}$. Conversely, if $\chi$ is not almost
$p$-rational then $\lev(\chi)\geq 2$, and it follows from Conjecture~\ref{conj:main}
that $\lev(\chi_{\bN_G(D)})\geq 2$.
\end{proof}

\begin{consequence}
Let $G$ be a finite group with abelian Sylow $p$-subgroups. Assume
that Conjecture \ref{conj:main} holds. Then
$\lev(\chi)=\lev(\chi_{\bN_G(D)})$ for every $\chi\in\Irr(G)$ of
$p$-rationality level at least $2$ and $D$ a defect group of the
$p$-block of $G$ containing $\chi$.
\end{consequence}

\begin{proof} This follows from the solution of the ``if" implication of Brauer's height zero conjecture \cite{Kessar-Malle13}. 
\end{proof}

\begin{consequence}
Let $\chi$ be an irreducible $2$-height zero character of a finite
group $G$. Suppose that $\QQ(\chi)=\QQ(\sqrt{d})$ is a quadratic
number field, where $d\nequiv 1 (\bmod\, 4)$ is a square-free integer.
Assume that Conjecture \ref{conj:main} holds. Then
$\QQ(\chi)=\QQ(\chi_{\bN_G(D)})$, where $D$ a defect group of the
$p$-block of $G$ containing $\chi$.
\end{consequence}

\begin{proof} Note that $c(\sqrt{d})=4|d|$ when $d\nequiv 1 (\bmod 4)$ is a square-free integer.
Therefore, by the hypothesis, $\lev(\chi)\geq 2$. By Conjecture
\ref{conj:main}, we then have $\lev(\chi_{\bN_G(D)})\geq 2$. In
particular, $\chi_{\bN_G(D)}$ is not rational, implying that
$\QQ(\chi)=\QQ(\chi_{\bN_G(D)})$.
\end{proof}


\subsection{Examples}\label{examples} We end with examples
to justify some of our claims from the Introduction.

First, for height-zero characters $\chi$ in general, $\lev(\chi)$
does not always equal $\lev(\chi_{D})$, where $D$ is a defect group
of the $p$-block containing $\chi$. For example, when $p=2$, the
group $\texttt{SmallGroup}(24,4)$ has characters of degree 2 with
$\lev(\chi)=2$ and $\lev(\chi_D)=0$; the group
\texttt{SmallGroup}(48,5) has characters of degree 2 with
$\lev(\chi)=3$ and $\lev(\chi_D)=2$; and several similar examples
can be found among the \texttt{SmallGroup} library in
\cite{GAP}. There are also examples for odd primes: when $p=3$, the group
\texttt{SmallGroup}(108,19) has characters of degree 3 with
$\lev(\chi)=2$ and $\lev(\chi_D)=1$.

Next, the assumption $\lev(\chi)\geq 2$ in
Conjecture~\ref{conj:main} is necessary. For example, the group
$2.\Al_{10}.2$ with $p=5$ has characters with degree $432$  with
$\lev(\chi)=1$ but $\lev(\chi_P)=0=\lev(\chi_{\bN_G(P)})$. Further,
$2.\Al_{11}$ with $p=3$ has characters of degree 1584 in blocks
having non-maximal defect with $\lev(\chi)=1$ and
$\lev(\chi_D)=0=\lev(\chi_{\bN_G(D)})$.


\end{document}